\documentclass[12pt]{article}
\usepackage{amsmath,amsfonts,amssymb,amsthm}
\usepackage[british]{babel}
\usepackage{color}
\usepackage{hyperref}
\usepackage{dsfont}

\newcommand{\F}{\mathcal{F}}
\newcommand{\calF}{\mathcal{F}}

\newcommand{\calS}{\mathcal{S}}
\newcommand{\calJ}{\mathcal{J}}
\newcommand{\pS}{\partial \mathcal{S}}
\newcommand{\R}{\mathbb{R}}

\newcommand{\DIV}{\textnormal{div}\,}
\newcommand{\eps}{\varepsilon}
\newcommand{\bbR}{\mathbb{R}}
\newcommand{\bbI}{\mathbb{I}}
\newcommand{\lra}{\longrightarrow}

\def\longrightharpoonup{\relbar\joinrel\rightharpoonup}
\def\cv{\stackrel{w}{\longrightharpoonup}}

\allowdisplaybreaks

\newtheorem{Theorem}{Theorem}
\newtheorem{Definition}{Definition}

\newtheorem{Proposition}{Proposition}
\newtheorem{Lemma}{Lemma}

\begin{document}
	
\date{\today}
\title{On the velocity of a small rigid body in a viscous incompressible fluid in dimension two and three}
\author{Marco Bravin \footnote{Delft Institute of Applied Mathematics, Delft University of Technology, Mekelweg 4
2628 CD Delft, The Netherlands }, \v{S}\'arka Ne\v{c}asov\'a\footnote{Institute of Mathematics, Czech Academy of Sciences \v{Z}itn\'a 25, 115 67 Praha 1, Czech Republic.}}

\maketitle
	
\begin{abstract}
In this paper we study the evolution of a small rigid body in a viscous incompressible fluid, in particular we show that a small particle is not accelerated by the fluid in the limit when its size converges to zero under a lower bound on its mass. This result is based on a new a priori estimate on the velocities of the centers of mass of rigid bodies that holds in the case when their masses are also allowed to decrease to zero. 
\end{abstract}

\section{Introduction}
	
In the recent works \cite{He:2D} and \cite{He:3D}, the authors showed that the presence of a small rigid body is negligible in a viscous incompressible fluid. In this paper we study the trajectory of this small object. In particular we will show that under some constrain on the mass, the small rigid body is not influenced by the fluid, in particular it will not be accelerated and will move with its initial velocity.  

Let start by introducing the equations that describe  the dynamic of the system. For $ d = 2, 3 $, let denote by $ \calS(t) \subset \bbR^{d} $  the position of the rigid body at time $ t \in \bbR^{+} = [0,+\infty).$  The fluid then occupies the domain $\calF(t) = \bbR^{d} \setminus \calS(t) $ and it is modelled by the Navier-Stokes equations

\begin{align}
\partial_t u +  u \cdot \nabla  u  - \nu \Delta u + \nabla p & \, =  0   \quad \quad && \text{ for } x \in \F(t), \nonumber \\
\DIV(u) & \, =  0 \quad \quad &&  \text{ for } x \in \calF(t),  \label{equ:fluid} \\
u & \, = u_{\calS} \quad \quad  && \text{ for } x \in \pS(t), \nonumber \\
u & \, \longrightarrow 0 && \text{ as } |x| \longrightarrow +\infty, \nonumber
\end{align}  
where $ u: \bbR^+ \times \calF(t) \lra \bbR^{d} $ is the velocity field and $ p: \bbR^+ \times  \calF(t) \lra \bbR $ is the pressure which is a scalar quantity. The real number $ \nu > 0 $ is the viscosity coefficient. Finally $ u_{\calS} $ is the velocity of the rigid body.      

Regarding the rigid body we assume that  it occupies $ \calS(0) = \calS^{in}  $  a closed, connected, simply connected subset of $ \bbR^{d} $ with no empty interior and  smooth boundary  and that it has density $ \rho^{in} : \calS^{in} \lra \bbR $ such that $ \rho^{in} > 0 $.  The dynamic of $ \calS(t) $ is completely determined by the evolution of the center of mass $ h(t) $ and the angular rotation $ Q(t) $ around the center of mass. More precisely 
\begin{equation*}
\calS(t) = \left\{ x \in \bbR^{d} \text{ such that }   Q^T(t)(x-h(t)) \in \calS^{in} \right\}.
\end{equation*}  
Here for a matrix $ A $ we denote by $ A^T $ its transpose. The density of the rigid body $ \calS(t) $ is given by $ \rho(t,x) = \rho^{in} (Q^T(t)(x-h(t)) ) $ and its velocity $ u_{\calS} : \bbR^{+}\times \calS(t) \lra \bbR^d$  is
\begin{align*}
u_{\calS}(t,x) = \frac{d}{dt}\left(   h(t) + Q(t) y   \right)\Big|_{y = Q^T(t)(x- h(t)) } = h'(t) + Q'(t)Q^T(t)(x-h(t)).
\end{align*}
The matrix $ Q(t) $ is a rotation, we deduce that $ Q'(t)Q^T(t) $ is skew-symmetric and can be identify in dimension three with a vector $ \omega(t) $ in the following way
\begin{equation*}
 Q'(t)Q^T(t) x = \omega(t) \times  x
\end{equation*} 
for any $ x \in \bbR^3 $. We will call $ \omega $ to the angular velocity. If we denote by $ \ell(t) = h'(t) $ the velocity of the center of mass the  solid velocity is $$ u_{\calS}(t,x) = \ell(t) + \omega(t) \times (x-h(t))  .$$ 

 Let recall that the mass $ m$ and the center of mass $ h $ is defined as
\begin{equation*} 
m = \int_{\calS(t) } \rho(t,x) \, dx \quad \text{ and } \quad h(t) = \frac{1}{m}\int_{\calS(t)} \rho(t,x) x \, dx  
\end{equation*}
and without loss of generality we assume $ h(0) = 0 $ and $ Q(0) = 0 $. The evolution of the center of mass $ h(t) $ and $ Q(t) $ follows the Newtons laws  that writes
\begin{align}
m \ell'(t)  = - \oint_{\pS(t)} \Sigma(u, & p)   nds,   \label{equ:rigid:body} \\
\calJ(t) \omega'(t) = \calJ(t) \omega(t) \times \omega(t) - & \oint_{\pS(t)} (x-h(t))\times \Sigma(u,p) n ds, \nonumber 
\end{align}
where $ n $ is the normal component to $ \partial \calF(t) $ exiting from the fluid domain, the inertia matrix $ \calJ(t) $ is defined as
\begin{equation*}
\calJ(t) = \int_{\calS(t)}  \rho(t,x) \left[ |x-h(t)|^2 \bbI -(x-h(t))\otimes (x-h(t))\right] \, dx  
\end{equation*} 
where $ \bbI $ is the identity matrix of $ d $ dimensions. Finally the stress tensor
\begin{equation*}
\Sigma(u,p) = 2 \nu D(u) - p \bbI \quad \text{ where } \quad D(u) = \frac{\nabla u + (\nabla u)^T }{2}
\end{equation*}
is the simmetric gradient.

For the system \eqref{equ:fluid}-\eqref{equ:rigid:body} the initial conditions are 
\begin{equation}
\label{equ:in:cond}
u(0,.) = u^{in} \quad \text{ in } \calF(t), \quad \ell(0) = \ell^{in} \quad \text{ and } \quad \omega(0) = \omega^{in}.
\end{equation}
where $  u^{in} $ is the initial velocity, $ \ell^{in} \in \bbR^{d} $ and $ \omega^{in} \in \bbR^{2d-3} $. Moreover they satisfy the compatibility conditions 
\begin{equation}
\label{comp:cond}
 \DIV(u^{in}) = 0 \quad \text{ in } \calF(0) \quad \text{ and } \quad u^{in} \cdot n = \ell^{in} + \omega^{in} \times x \quad \text{ on } \partial \calS(0).  
\end{equation}
Notice that in the case of dimension two the matrix $  Q'(t)Q^T(t) $ can be identify with a scalar quantity that we denote again by $ \omega $ and the $ u_{\calS}(t,x) = \ell(t) + \omega(t) (x-h(t))^{\perp}  $, where for $ x \in \bbR^2 $ we denote by $ x^{\perp} = (-x_2, x_1 )^T $. Moreover the inertia matrix $ \calJ $ becomes a scalar independent of time and the second equation of  \eqref{equ:rigid:body} simplifies. 

Let recall that the system \eqref{equ:fluid}-\eqref{equ:rigid:body}-\eqref{equ:in:cond} has been widely studied in the literature. In fact the first works on the existence of Hopf-Leray type weak solutions  are \cite{Jud74} and  \cite{Serre} where they consider the case $ \Omega = \bbR^3 $. These results were then extended in \cite{GLS00}-\cite{CST}-\cite{DE}-\cite{Fei}-\cite{Fei1}.  Uniqueness was shown in \cite{GS15}  in dimension two and  in \cite{MNR} in dimension three under Prodi-Serrin conditions. Regularity was studied in dimension three under Prodi-Serrin conditions in \cite{MNR1}.  Well-posedness of strong solutions in Hilbert space setting was proved in \cite{GM}-\cite{Tak}-\cite{TT} and in the Banach space setting in \cite{GGH}-\cite{MT}. Notice that similar results holds in the case the Navier-Slip boundary conditions are prescribed on $\partial \calS $, see \cite{PS}-\cite{GH}-\cite{IO1}-\cite{BGMN}. See also the problem of weak-strong uniqueness \cite{CNM}.

Among all the different types of solutions, in this work we consider Hopf-Leray type weak solutions for the system \eqref{equ:fluid}-\eqref{equ:rigid:body}-\eqref{equ:in:cond}. To recall this definition let introduce some notations from \cite{He:3D}. Let denote by 
\begin{equation}
 \tilde{\rho} = \chi_{\calF(t)} + \rho\chi_{\calS(t)}
\end{equation}
which is the extension by $ 1 $ of the density of the rigid body. Here for a set $ A \subset \bbR^{d} $, we denote by $ \chi_{A} $ the indicator function of $ A $, more precisely $ \chi_A(x) = 1 $ for $ x \in A $ and $ 0 $ elsewhere. Similarly we define the global velocity 
\begin{equation*}
\tilde{u} = u \chi_{\calF(t)} +u_{\calS} \chi_{\calS(t)} = u \chi_{\calF(t)} + \left(\ell(t) + \omega \times (x-h(t)) \right) \chi_{\calS(t)}.
\end{equation*}
Notice that if $ \tilde{u}^{in} \in L^{2}(\calF(0))$, then the compatibility conditions \eqref{comp:cond} on the initial data imply that $ \DIV( \tilde{u}^{in}) = 0 $ in an appropriate weak sense. After all this preliminary we introduce the definition of Hopf-Leray type weak solutions for the system \eqref{equ:fluid}-\eqref{equ:rigid:body}-\eqref{equ:in:cond}.

\begin{Definition}

Let $ \calS^{in} $ and $ \rho^{in} $ the initial position and density of the rigid body , let $ (u^{in}, \ell^{in}, \omega^{in}) $ satisfying the hypothesis \eqref{comp:cond} and such that $ \tilde{u}^{in } \in L^2(\bbR^d) $. Then a triple $ ( u, \ell, \omega ) $ is a Hopf-Leray weak solution for the system \eqref{equ:fluid}-\eqref{equ:rigid:body}-\eqref{equ:in:cond} associated with the initial data  $ \calS^{in} $, $ \rho^{in} $, $ u^{in} $, $\ell^{in}$ and $ \omega^{in} $, if 
\begin{itemize}

\item the functions $ u $, $ \ell $ and $ \omega $ satisfy 
\begin{gather*}
\ell \in L^{\infty}(\bbR^+;\bbR^d), \quad \omega \in L^{\infty}(\bbR^+; \bbR^{2d-3}) \\
u \in L^{\infty }(\bbR^+;L^{2}(\calF(t))) \cap L^2_{loc}(\bbR^+; H^{1}(\calF(t))), \quad \text{and} \quad \tilde{u} \in C_w(\bbR^+; L^2(\bbR^d));
\end{gather*}

\item the vector field $ \tilde{u} $ is divergence free in $ \bbR^d $ with $ D(\tilde{u}) = 0 $ in $ \calS(t) $; 

\item the vector field $ \tilde{u} $ satisfies the equation in the following sense:
\begin{equation}
\label{weak:form}
-\int_{\bbR^+} \int_{\bbR^d} \tilde{\rho} \tilde{u} \cdot (\partial_t \varphi + \tilde{u} \cdot \nabla \varphi ) - 2 \nu D(\tilde{u}): D(\varphi) \, dx \, dt = \int_{\bbR^d} \tilde{\rho}^{in} \tilde{u}^{in} \cdot \varphi(0,.) \, dt,
\end{equation}
for any test function $ \varphi \in C^{\infty}_{c}(\bbR^+ \times \bbR^d ) $ such that $ \DIV(\varphi) = 0 $ and $ D(\varphi) = 0 $ in $ \calS(t )$. 

\item The following energy inequality holds
\begin{equation}
\label{ene:ineq}
\int_{\bbR^{d}} \tilde{\rho}(t,.) |\tilde{u}(t,.)|^2 \, dx + 4 \nu \int_{0}^t \int_{\bbR^d} |D(\tilde{u})| ^2\, dx dt \leq \int_{\bbR^d} \tilde{\rho}|\tilde{u}^{in}|^2,
\end{equation}
for almost any time $ t \in \bbR^+ $.

\end{itemize}

\end{Definition}

There exists of weak solution for the system \eqref{equ:fluid}-\eqref{equ:rigid:body}-\eqref{equ:in:cond} is now classical and can be find for example in \cite{Fei}.

\begin{Theorem}
For initial data  $ \calS^{in} $, $ \rho^{in} $, $ u^{in} $, $\ell^{in}$ and $ \omega^{in} $ satisfying the hypothesis \eqref{comp:cond} and such that $ \tilde{u}^{in } \in L^2(\bbR^d) $, there exist a Hopf-Leray weak solution $ ( u, \ell, \omega ) $ of the system \eqref{equ:fluid}-\eqref{equ:rigid:body}-\eqref{equ:in:cond}.
\end{Theorem}

Let now introduce a small parameter $ \eps > 0 $ that control the size of the rigid body. We will consider initial rigid body of the size $ \calS^{in}_{\eps} \subset B_{0}(\eps) $. In \cite{He:2D} and \cite{He:3D} the authors studied the limit as $ \eps $ goes to zero for solutions of the system \eqref{equ:fluid}-\eqref{equ:rigid:body}-\eqref{equ:in:cond} under some mild assumption on the initial data $ \rho^{in}_{\eps} $, $ u^{in}_{\eps} $, $\ell^{in}_{\eps} $ and $ \omega^{in}_{\eps} $ and in particular they show that in the limit the the presence of the rigid body does not influence the limit dynamics.  These results can be resume as follow. 

\begin{Theorem}[Th. 3 of \cite{He:2D} and Th. 2 of \cite{He:3D}.]
\label{theo:he}
Let $ (u_{\eps}, \ell_{\eps}, \omega_{\eps}) $ a sequence of Hopf-Leray solutions associated with the initial data $ \calS^{in}_{\eps} $, $ \rho^{in}_{\eps} $, $ u^{in}_{\eps} $, $\ell^{in}_{\eps} $ and $ \omega^{in}_{\eps} $  satisfying the hypothesis \eqref{comp:cond} and such that $ \tilde{u}^{in }_{\eps} \in L^2(\bbR^d) $. If we assume that 
\begin{itemize}

\item the rigid body $ \calS^{in}_{\eps} \subset B_{0}(\eps) $;

\item the mass of the rigid body $ m_{\eps}/ \eps^{d} \lra +\infty $;

\item the initial velocity $ \tilde{u}^{in}_{\eps} \lra u^{in} $ in $ L^{2}(\bbR^{d}) $ and $ m_{\eps} |\ell_{\eps}|^{2} + (\calJ_{\eps} \omega_{\eps}^{in}) \cdot \omega_{\eps}^{in} \lra2 E $;

\end{itemize}  
Then up to subsequence 
\begin{equation*}
\tilde{u}_{\eps} \cv u \quad \text{ weak-}\star \text{ in } L^{\infty}(\bbR^+; L^2(\bbR^d))   \text{ and weak  in }  L^{2}_{loc}(\bbR^+; H^1(\bbR^d))
\end{equation*}
where $ u $ is a distributional solution to the Navier-Stokes equations that satisfies the energy inequality
\begin{equation}
\label{ene:ineq:plus:extra}
\int_{\bbR^{d}} | u(t,.)|^2 \, dx + 4 \nu \int_{0}^t \int_{\bbR^d} |D(u)| ^2\, dx dt \leq \int_{\bbR^d} |u^{in}|^2 + E.
\end{equation}

\end{Theorem}

Notice that there are many results in this direction. For example in \cite{LT} the authors studied the vanishing object problem in dimension two under the assumption $ m_{\eps} = \eps^2 m $. For viscous compressible fluid the three dimensional case was tackled in \cite{ION} and improved in \cite{FRZ2} and more recently in \cite{FRZ} was studied the two dimensional case for a weakly compressible fluid. Notice that for viscous compressible fluid the vanishing object problem in two dimensions is still an open problem due to the lack of integrability of the pressure.       

In the case of dimension two and for an inviscid incompressible fluid modelled by the Euler equations, i.e. system  \eqref{equ:fluid}-\eqref{equ:rigid:body}-\eqref{equ:in:cond} with $ \nu = 0 $, the vanishing object problem was studied in \cite{GMS}-\cite{GLS1}-\cite{GLS2}. In this case the presence of the small rigid body creates a vortex point in the limiting dynamics and the intensity is associated with the initial circulation around the object. In this case the author where able to determine the position of the rigid body in the limit and it coincides with the center of the vortex.

The goal of this paper is to study the evolution of the small rigid body in the limit as $\eps \lra 0 $ in the case the fluid is viscous and incompressible. In particular we will show that its dynamics is not influenced by the fluid.

\begin{Theorem}
\label{main:theo}
Let $ (u_{\eps}, \ell_{\eps}, \omega_{\eps}) $ a sequence of Hopf-Leray solutions associated with the initial data $ \calS^{in}_{\eps} $, $ \rho^{in}_{\eps} $, $ u^{in}_{\eps} $, $\ell^{in}_{\eps} $ and $ \omega^{in}_{\eps} $  satisfying the hypothesis \eqref{comp:cond} and such that $ \tilde{u}^{in }_{\eps} \in L^2(\bbR^d) $.  If to the assumptions of Theorem \ref{theo:he} we add 

\begin{itemize}

\item $ m_{\eps}/ \eps^{1/2} \lra +\infty $ for $ d = 3$ and $ m_{\eps} \geq C > 0 $ for $ d = 2$;

\item $ \ell_{\eps}^{in} \lra \ell^{in} $.

\end{itemize}  
Then 
\begin{equation*}
\ell_{\eps} \lra \ell^{in} \quad \text{ in } L^{\infty}_{loc}(\bbR^+).
\end{equation*}

\end{Theorem}

Notice that in the above theorem $ \ell_{\eps} \lra \ell^{in} $ and   $  \ell^{in}  $ is independent of time. This means that the small rigid body in not accelerated by the fluid.

 The difficulty of the result is the fact that we allow the mass of the rigid body to go to zero, in particular it is not enough to show that in the limit $ m \ell' = 0 $ because $ m = 0 $. Moreover from the energy estimate we only deduce $ m_{\eps} |\ell_{\eps}|^2 $ uniformly bounded and this is not enough to have an \textit{a priori} bound on $ \ell_{\eps} $.

Finally notice that the evolution of the small rigid body seems richer in the case of a two dimensional inviscid incompressible fluid but this is due to the fact that in this setting it is possible to consider initial data with non-zero circulation around the object. In the case of zero circulation, the limit velocity of the center of mass of the small rigid body is trivial in the sense that $ \ell(t) = \ell^{in} $, see Section 1.4 of \cite{Munnier}. Notice that a velocity field that has non-zero circulation around the body behaves as $ 1/|x| $ when $ | x |  $ goes to $ +\infty $, in particular it is not $ L^2 $ so it does not enter in the theory of Hopf-Leray weak solutions, see also the comments in section 2 of \cite{LT}. Let recall that an existence result in this direction is available in \cite{IO2} where the author considers initial data for the velocity field of the type $ u^{in}+ x^{\perp}/|x^2| $ with $ u^{in} \in L^2(\calF^{in}) $ for the system   \eqref{equ:fluid}-\eqref{equ:rigid:body}-\eqref{equ:in:cond}. The vanishing rigid body problem is still open in this setting.  

Let now move to the proof of Theorem \ref{main:theo}. The remaining part of the paper is divided in two main sections. First of all we recall some useful cut off. In the second one we show the $ L^{\infty} $ convergence for the velocity of the center of mass. 
\section{Some appropriate cut-off}

In this section we introduce some cut-off that has been considered also in \cite{He:2D} and \cite{ION}. This cut-off has the property that they optimized the $ L^{d} $ norm of the gradient. First of all for $ A, B \in \R $ with $ 0 < A < B $, we denote by $ \alpha = B/ A > 1 $ and we define the functions
\begin{equation*}
f_{A,B}(z) = \begin{cases}
1 \quad & \text{ for } 0\leq z < A, \\
\frac{\log z - \log B }{\log A - \log B } \quad & \text{ for } A \leq z \leq B, \\
0 \quad & \text{ for } z > B.
\end{cases}
\end{equation*} 
It holds that $ f_{A,B} \in W^{1,\infty}(\R^{+}) $. We define the d-dimensional cut-off 
\begin{equation*}
\eta_{\eps, \alpha_{\eps}}(x) = f_{\eps, \alpha_{\eps} \eps }(|x|),  
\end{equation*} 
where $ \alpha_{\eps} $ will be choose appropriately.

\begin{Proposition}
Under the hypothesis that $ \eps \alpha_{\eps} \to 0$, it holds

\begin{enumerate}

\item The functions $  \eta_{\eps, \alpha_{\eps}} \longrightarrow 0 $  in $ L^{q}(\R^d) $ for $ 1 \leq q < +\infty $.

\item We have $$ \left\| \nabla \eta_{\eps, \alpha_{\eps}} \right\|_{L^{d}(\R^d)}^d = \frac{2^{d-1}\pi}{(\log \alpha_{\eps})^{d-1} }. $$

\item For $ 1\leq q < d $, $$ \left\| \nabla \eta_{\eps, \alpha_{\eps}} \right\|_{L^{q}(\R^3)}^q = \frac{2^{d-1}\pi}{d-q}  \frac{\alpha_{\eps}^{d-q}-1}{(\log \alpha_{\eps})^{q} } \eps^{d-q}. $$

\end{enumerate}

\end{Proposition}

\begin{proof}
After passing to spherical coordinates the proof is straight-forward. For example to show part $ 2. $ in dimension three, we compute
\begin{align*}
 \left\| \nabla \eta_{\eps, \alpha_{\eps}} \right\|_{L^3(\bbR^3)}^3 = & \, \int_0^{2\pi} \int_{0}^{\pi} \int_{\eps}^{\eps \alpha_{\eps}} \left| \frac{1}{r} \frac{1}{\log (\eps) - \log(\eps \alpha_{\eps})}\right|^3 \sin(\varphi) r^2  \, dr d\varphi d\theta  \\
= & \, \frac{4 \pi}{(\log(\alpha_{\eps}))^3.}\left[\log(r)\right]_{\eps}^{\eps \alpha_{\eps}} \\
= & \, \frac{4 \pi}{(\log(\alpha_{\eps}))^2.}
\end{align*} 
See Lemma 2 of \cite{He:2D} for the proof in dimension two.
\end{proof}

In the same spirit as in \cite{He:2D}, we introduce a cut-off that follows the rigid body and that is divergence free. In this way we can use it as an admissible test function in \eqref{weak:form}.

In dimension $ d = 3 $, for $ x, z \in \bbR^3 $ we introduce
\begin{equation}
\label{cut:off:3D}
\Psi_{\eps, \alpha_{\eps}}[z](x) = \text{curl} \left( \eta_{\eps, \alpha_{\eps}}(x) \begin{pmatrix} z_2 x_3 \\ z_3 x_1 \\ z_1 x_2 \end{pmatrix} \right) . 
\end{equation}
Similarly  for $ x, z \in \bbR^2  $ we introduce
\begin{equation}
\label{cut:off:2D}
\Psi_{\eps, \alpha_{\eps}}[z](x) = \nabla^{\perp} \left( \eta_{\eps, \alpha_{\eps}}(x) (z_2 x_1 -z_1 x_2 ) \right) . 
\end{equation}
Notice that by definition $ \Psi_{\eps, \alpha_{\eps}}[z] $ is divergence free and equal to $ z $ in $ B_{\eps}(0) $.  Moreover they satisfy the same bounds of $ \eta_{\eps, \alpha_{\eps}} $.

\begin{Lemma}
\label{lem:conv}
Under the hypothesis that $ \eps \alpha_{\eps} \to 0$, it holds

\begin{enumerate}

\item For $ 1 \leq q < +\infty $, we have $$ \|\Psi_{\eps, \alpha_{\eps}}[z](x)  \| \leq C (\eps \alpha_{\eps})^{d/q}. $$ 

\item We have $$ \left\| \nabla \Psi_{\eps, \alpha_{\eps}}[z](x) \right\|_{L^{d}(\R^d)}^d \leq  C \frac{|z|}{(\log \alpha_{\eps})^{d-1} }. $$

\item For $ 1\leq q < d $, $$ \left\| \nabla \Psi_{\eps, \alpha_{\eps}}[z](x) \right\|_{L^{q}(\R^3)}^q = C\frac{|z|}{d-q}  \frac{\alpha_{\eps}^{d-q}-1}{(\log \alpha_{\eps})^{q} } \eps^{d-q}. $$

\end{enumerate}

\end{Lemma}

\begin{proof}

Notice that  $ \Psi_{\eps, \alpha_{\eps}}[z] = \eta z + G \nabla \eta  $ and $ \nabla \Psi_{\eps, \alpha_{\eps}}[z]  = \nabla \eta + \tilde{G} \nabla \eta + G: \nabla^2 \eta $   where  
\begin{equation*}
G = \begin{pmatrix} 0 & z_1 x_2 & z_{3} x_1 \\ -z_1 x_2 & 0 & z_2 x_3 \\ -z_3 x_1 & -z_2 x_3 & 0 \end{pmatrix}  \quad \text{ and } \quad  \tilde{G} = \begin{pmatrix} 0 & z_1  & z_{3}  \\ -z_1  & 0 & z_2  \\ -z_3  & -z_2  & 0 \end{pmatrix}. 
\end{equation*}
Using that $ |G(r, \theta, \phi )| \leq C/r  $ where $ (r, \theta, \phi ) $ are polar coordinates, 
\begin{gather*}
|\nabla \eta_{\eps, \alpha_{\eps}}(r)| \leq C \frac{1}{r}\frac{1}{\log (\eps)- \log(\eps \alpha_{\eps})} \quad \text{ and } \quad |\nabla^2 \eta_{\eps, \alpha_{\eps}}(r)| \leq C \frac{1}{r^2}\frac{1}{\log (\eps)- \log(\eps \alpha_{\eps})} .
\end{gather*} 
After straight-forward integration we derive the proof.

\end{proof}

We use the special test functions \eqref{cut:off:3D} and \eqref{cut:off:2D} in equation \eqref{weak:form} to prove the main result.

\section{Proof of Theorem \ref{main:theo}}

In this section we will prove Theorem \ref{main:theo} in dimension two and three.  From the definition of Leray-Hopf weak solutions and in particular from \eqref{ene:ineq}, we have 
\begin{align*}
m_{\eps}  |\ell_{\eps}(t)|^2 + \left( \calJ_{\calS_{\eps}(t)} \omega_{\eps}(t)\right) \cdot \omega_{\eps}(t) +  \int_{\calF_{\eps}(t)} |u_{\eps}(t,.)|^2 \, dx + 4 \nu \int_{0}^t \int_{\bbR^d}& |D(u_{\eps})|^2\, dx dt \\ & \,  \leq \int_{\bbR^d} \tilde{\rho}^{in}_{\eps}|\tilde{u}^{in}_{\eps}|^2 \leq C,
\end{align*}
with $ C $ independent of $ \eps $. The above estimates was used in \cite{He:2D} and \cite{He:3D} to show Theorem \ref{theo:he}. This estimate does not give information on $ \ell_{\eps} $ in the case the mass of the rigid body converges to zero. We will now present a new estimates that gives us control of the $ L^{\infty}$-norm of $ \ell_{\eps} $.
Recall that that the $ L^{\infty} $-norm in the interval $ (0,T) $ can be define by 
\begin{equation*}
 \|f\|_{L^{\infty}(0,T)} = \sup_{ \gamma \in C^{\infty}_{c}(0,T)} \frac{\left| \int_{0}^T f \gamma \, d\tau \right|}{\|\gamma\|_{L^1(0,T)}}.
 \end{equation*}
We will use this form to show the results.

Given $ \gamma \in C^{\infty}_{c}([0,T))$, we consider the divergence free vector fields
\begin{equation*}
\psi_{\eps}(t,x) = \frac{1}{m_{\eps}}\Psi_{\eps, \alpha_{\eps}} \left[ \int_t^T \gamma(\tau) \, d\tau   \right](x-h_{\eps}(t))
\end{equation*}
for $ t \in [0,T] $ and $ \psi_{\eps} $ identically zero for $ t > T $.
First of all notice that for $ x \in \calS_{\eps}(t) $
\begin{equation*}
\psi_{\eps}(t,x) = \frac{1}{m_{\eps}}\Psi_{\eps , \alpha_{\eps}} \left[ \int_t^T \gamma(\tau) \, d\tau  \right](x-h_{\eps}(t)) =  \int_t^T \gamma(\tau) \, d\tau.
\end{equation*}
We deduce that 
\begin{align*}
- \int_{\bbR^+} \int_{\calS_{\eps}(t)} \rho u_{\calS} \cdot \partial_t \psi_{\eps} \, dx dt = \int_0^T \ell_{\eps} \cdot \gamma  \, dt,
\end{align*}
and similarly
\begin{equation*}
\int_{\calS_{\eps}(0)} \rho u_{\calS}^{in} \cdot \psi_{\eps}(0,.) \, dx = \ell_{\eps}^{in} \cdot \int_{0}^{T} \gamma(\tau) \, d\tau.
\end{equation*}

To show the $ L^{\infty} $ convergence of $ \ell_{\eps} $ is then enough to use the equation \eqref{weak:form} rewritten in the form
\begin{align}
\label{equ:to:est}
 \int_0^T (\ell_{\eps} - \ell^{in}) \cdot \gamma \, dt = & \,
\int_{0}^T \int_{\calF_{\eps}(t)}  u_{\eps} \cdot (\partial_t \psi_{\eps} + u_{\eps} \cdot \nabla \psi_{\eps} ) - 2 \nu D(u_{\eps}): D(\psi_{\eps}) \, dx \, dt \\ & \, + \int_{\calF_{\eps}(0)}  u^{in}_{\eps} \cdot \psi_{\eps}(0,.) \, dx + (\ell^{in}_{\eps}-\ell^{in})\cdot \int_{0}^{T} \gamma \, dt . \nonumber
\end{align}
In particular, if we show that the absolute values of the right hand side of \eqref{equ:to:est} is bounded by
\begin{equation*}
c(\eps) \|\gamma\|_{L^{1}(0,T)}
\end{equation*}
with $ c(\eps) \lra 0 $, then the convergence $ \ell_{\eps} \lra \ell^{in} $ follows from \eqref{equ:to:est}.

It remains to estimate the right hand side of the above inequality. Let start by computing for $ x \in \calF_{\eps}(t) $ 
\begin{align}
\label{expr:der:time}
\partial_t \psi_{\eps} = & \, \partial_t \left( \frac{1}{m_{\eps}}\Psi_{\eps, \alpha_{\eps}} \left[ \int_t^T \gamma (\tau) \, d\tau  \right](x-h_{\eps}(t)) \right) \\
= & \,  -\frac{1}{m_{\eps}}\Psi_{\eps, \alpha_{\eps}} \left[ \gamma(t)   \right](x-h_{\eps}(t)) - 	\ell_{\eps} \cdot \nabla  \left( \frac{1}{m_{\eps}}\Psi_{\eps} \left[ \int_t^T \gamma(\tau) \, d\tau  \right](x-h_{\eps}(t)) \right) \nonumber
\end{align}
At this step the proof in dimension two and three start to differ so let start by consider the case of dimension three.

\begin{proof}[Proof of Theorem \ref{main:theo} for $ d = 3 $.]
 In this case, we choose $ \alpha_{\eps} = 2 $. For simplicity we write $ \Psi_{\eps, 2} = \Psi_{\eps} $.
 Let estimate any term of the right hand side of \eqref{equ:to:est} separately.
 Using the computation \eqref{expr:der:time}, we deduce that 
 \begin{align*}
\left|  \int_{0}^T \int_{\calF_{\eps}(t)} u_{\eps} \cdot \partial_t \psi_{\eps} \, dx dt \right|  \leq & \,  \left| \int_{0}^T \int_{\calF_{\eps}(t)} u_{\eps} \cdot \frac{1}{m_{\eps}}\Psi_{\eps} \left[ \gamma(t)   \right](x-h_{\eps}(t)) \, dx dt  \right| \\
& \, +  \left|  \int_{0}^T \int_{\calF_{\eps}(t)} u_{\eps} \cdot \left[ \ell_{\eps} \cdot \nabla  \left( \frac{1}{m_{\eps}}\Psi_{\eps} \left[ \int_t^T \gamma(\tau) \, d\tau  \right](x-h_{\eps}(t)) \right) \right]	\, dx dt \right| \\
\leq & \, \|u_{\eps}\|_{L^{\infty}(0,T;L^{2}(\calF_{\eps}(t)))} \left\| \frac{\Psi_{\eps}}{m_{\eps}} \right\|_{L^1(0,T;L^{2}(\bbR^3))}	\\ 
& \, + \| u_{\eps} \|_{L^{\infty}(0,T;L^{2}(\calF_{\eps}(t)))} \left\|\frac{\nabla \Psi_{\eps}}{m_{\eps}} \right\|_{L^1(0,T;L^2(\bbR^3))} \left\| \int_0^T \ell_{\eps} \right\|_{L^{\infty}(0,T)} \\
\leq & \, \|u_{\eps}\|_{L^{\infty}(0,T;L^{2}(\calF_{\eps}(t)))} \|\Psi_{\eps}[1]\|_{L^{3}(\bbR^3)} \frac{\eps^{1/2}}{m_{\eps}} \|\ell_{\eps}\|_{L^1(0,T)} \\
& \, +   \|u_{\eps}\|_{L^{\infty}(0,T;L^{2}(\calF_{\eps}(t)))} \|\nabla \Psi_{\eps}[1]\|_{L^{3}(\bbR^3)} \frac{\eps^{1/2}}{m_{\eps}} T \|\ell_{\eps}\|_{L^{1}(0,T)} \|\gamma\|_{L^{1}(0,T)}. 
\end{align*}
Then we consider 
\begin{align*}
\left|  \int_{0}^T \int_{\calF_{\eps}(t)} u_{\eps} \cdot (u_{\eps} \cdot \nabla \psi_{\eps}) \, dx dt \right|  \leq & \, \|u_{\eps}\|_{L^2(0,T;L^{6}(\calF_{\eps}(t)))}^2 \left\| \frac{\nabla \Psi_{\eps}}{m_{\eps}} \right\|_{L^{\infty}(0,T;L^{3/2}(\calF_{\eps}(t)))}  \\
\leq & \, \|\tilde{u}_{\eps}\|_{L^2(0,T;L^{6}(\bbR^3))}^2 \|\nabla \Psi_{\eps}[1]\|_{L^{3}(\bbR^3)} \frac{\eps}{m_{\eps}} \|\gamma \|_{L^1(0,T)},
\end{align*}
where we used that $ \|\nabla \tilde{u}_{\eps} \|_{L^2(\bbR^3)} \leq C \| D(\tilde{u}_{\eps})\|_{L^2(\bbR^3)} $ by the Korn inequality to deduce a uniform bound on $ \| \tilde{u}_{\eps} \|_{L^{2}(0,T;L^{6}(\bbR^3)} $ by Sobolev embedding.

Similarly 
\begin{align*}
\left|  \int_{0}^T \int_{\calF_{\eps}(t)} D u_{\eps} \cdot D \psi_{\eps} \, dx dt \right|  \leq & \, \|\nabla u_{\eps}\|_{L^{2}(0,T;L^{2}(\calF_{\eps}(t)))} \left\| \frac{\Psi_{\eps}}{m_{\eps}} \right\|_{L^{2}(0,T;L^{2}(\calF_{\eps}(t)))}  \\
\leq & \, \| \nabla u_{\eps}\|_{L^2(0,T;L^{2}(\calF_{\eps}(t)))}^2 \|\nabla \Psi_{\eps}[1]\|_{L^{3}(\bbR^3)} \frac{\eps^{1/2}}{m_{\eps}} \sqrt{T} \|\gamma \|_{L^1(0,T)}.
\end{align*}

Finally for the initial data we 
\begin{align*}
\left|  \int_{\calF_{\eps}(0)}  u^{in} \cdot \psi_{\eps} \, dx \right| \leq & \, \left|  \int_{\calF(0)}  \frac{1}{m_{\eps}} u^{in}_{\eps} \cdot \Psi_{\eps}\left[ \int_0^T \ell_{\eps} \right]  \, dx \right|  \\
\leq & \, \|u^{in}_{\eps}\|_{L^{2}(\calF(0))}\|\Psi_{\eps}[1]\|_{L^{3}(\bbR^3)} \frac{\eps^{1/2}}{m_{\eps}} \|\gamma\|_{L^1(0,T)},
\end{align*}
and 
\begin{align*}
\left|  (\ell_{\eps}^{in}- \ell^{in}) \cdot \int_{0}^{T} \gamma \, d\tau \right| \leq & \, \left| \ell_{\eps}^{in} - \ell^{in} \right| \|\gamma\|_{L^1(0,T)}
\end{align*}
Putting all the above estimates together and using the uniform estimates \eqref{ene:ineq}, the convergence of the initial data and the hypothesis that $ m_{\eps}/ \eps^{1/2} \lra +\infty $, we deduce that 
\begin{align*}
\left|  \int_0^T (\ell_{\eps} - \ell^{in}) \cdot \gamma \, dt  \right| \leq \tilde{c}(\eps) \|\gamma\|_{L^{1}} + \bar{c}(\eps) \|\ell_{\eps}\|_{L^{1}(0,T)} \|\gamma\|_{L^{1}(0,T)},
 \end{align*}
where $ \tilde{c}(\eps), \bar{c}(\eps) \lra 0 $ and 
\begin{gather}
\bar{c}(\eps) = \|u_{\eps}\|_{L^{\infty}(0,T;L^{2}(\calF_{\eps}(t)))} \|\nabla \Psi_{\eps}[1]\|_{L^{3}(\bbR^3)} \frac{\eps^{1/2}}{m_{\eps}} T
\end{gather}
If we divide the right and left hand side by the $ L^1$ norm of $ \gamma $  and we take the $ \sup $ in $ \gamma \in C^{\infty}_{c}(0,T)$, we deduce 
\begin{equation*}
\|\ell_{\eps} - \ell^{in} \|_{L^{\infty}} \leq \tilde{c}(\eps) + \bar{c}(\eps) \|\ell_{\eps} - \ell^{in} \|_{L^{\infty}(0,T)} + \bar{c}(\eps) |\ell^{in}|.
\end{equation*}
We can now absorb the second term of the right hand side and deduce that $ \ell_{\eps} \lra \ell^{in} $ in $ L^{\infty} $.

\end{proof}

Let now move to the case of dimension two. In this case the the energy inequality \eqref{ene:ineq} implies that $ \ell_{\eps} $ is uniformly bounded in $ L^{\infty}(0,T) $ thanks to the hypothesis under $ m_{\eps} \geq C > 0 $  and the assumption on the initial data. It remains to show the strong convergence.

\begin{proof}[Proof of Theorem \ref{main:theo} for $ d = 2$.] Choose $ \alpha_{\eps} \lra + \infty $ and consider test functions
\begin{equation*}
 \psi_{\eps}(x,t) = \frac{1}{m_{\eps}}\Psi_{\eps, \alpha_{\eps}} \left[ \int_{t}^{T} \gamma \, d\tau \right](x-h_{\eps}(t)).
\end{equation*}
for $ \gamma : [0,T] \lra \bbR^2  $ with $ \gamma \in C^{\infty}_{c}((0,T))$.
Notice that \eqref{equ:to:est} holds also in the case of dimension two, in fact the same computations are true. We only need to show that the right hand side of \eqref{equ:to:est} converges to zero. 

Let estimate for example 
\begin{align*}
\left|  \int_{0}^T \int_{\calF_{\eps}(t)} u_{\eps} \cdot \partial_t \psi_{\eps} \, dx dt \right|  \leq & \,  \left| \int_{0}^T \int_{\calF_{\eps}(t)} u_{\eps} \cdot \frac{1}{m_{\eps}}\Psi_{\eps,  \alpha_{\eps}} \left[ \gamma(t)   \right](x-h_{\eps}(t)) \, dx dt  \right| \\
& \, +  \left|  \int_{0}^T \int_{\calF_{\eps}(t)} u_{\eps} \cdot \left[ \ell_{\eps} \cdot \nabla  \left( \frac{1}{m_{\eps}}\Psi_{\eps,  \alpha_{\eps}} \left[ \int_t^T \gamma \, d\tau  \right](x-h_{\eps}(t)) \right) \right]	\, dx dt \right| \\
\leq & \, \|u_{\eps}\|_{L^{\infty}(0,T;L^{2}(\calF_{\eps}(t)))} \left\| \frac{\Psi_{\eps,  \alpha_{\eps}}[1]}{m_{\eps}} \right\|_{L^{2}(\bbR^2)} \|\gamma \|_{L^1(0,T)} 	\\ 
& \, + \| u_{\eps} \|_{L^{\infty}(0,T;L^{2}(\calF_{\eps}(t)))} \left\|\frac{\nabla \Psi_{\eps,  \alpha_{\eps}}[1]}{m_{\eps}} \right\|_{L^2(\bbR^2)} \|\gamma \|_{L^1} \sqrt{T}\left\| \ell_{\eps} \right\|_{L^{1}(0,T)}. 
\end{align*}
As before we can absorb the last term and the other converges to zero because
\begin{equation*}
\left\| \frac{\Psi_{\eps,  \alpha_{\eps}}}{m_{\eps}} \right\|_{L^{2}(\bbR^2)} + \left\|\frac{\nabla \Psi_{\eps,  \alpha_{\eps}}}{m_{\eps}} \right\|_{L^2(\bbR^2)} \lra 0 
\end{equation*}
by the hypothesis $ m_{\eps} \geq C $ and $ \alpha_{\eps} \lra +\infty $ due to part $ 2. $ of Lemma \eqref{lem:conv}. The convergence to zero of the other terms of the right hand side of \eqref{equ:to:est} follows analogously. We conclude that $ \ell_{\eps} \lra \ell^{in} $ in $ L^{\infty}(0,T)$.

\end{proof}

\section*{Acknowledgements}

{\small M.B. is supported by the NWO grant OCENW.M20.194.  {\v S}. N.  has been supported by  Praemium Academiæ of \v S. Ne\v casov\' a and by the Czech Science Foundation (GA\v CR) project 22-01591S. The Institute of Mathematics, CAS is supported by RVO:67985840. }

\end{document}